\documentclass[11pt]{amsart}
\usepackage{amsmath, amsthm, amssymb, amscd, graphicx, times, hyperref, epstopdf}
\usepackage[hmargin=1.2in,vmargin=1.2in]{geometry}
\newtheorem{thm}{Theorem}[section]
\newtheorem{cor}[thm]{Corollary}

\newtheorem{lem}[thm]{Lemma}

\theoremstyle{definition}
\newtheorem{defn}[thm]{Definition}

\theoremstyle{remark}
\newtheorem{rem}[thm]{Remark}

\newtheorem{qu}[thm]{Question}
\numberwithin{equation}{section}
\newtheoremstyle{dotless}{}{}{}{}{\bfseries}{}{ }{}
\theoremstyle{dotless}
\newtheorem*{repthm}{Theorem}
\newtheorem*{repcor}{Corollary}
\newtheorem*{replem}{Lemma}

\def\N{{\mathbb N}}

\def\R{{\mathbb R}}

\newcommand{\ra}{\rightarrow}

\begin{document}
\title {Hyperbolic quasi-geodesics in CAT(0) spaces}
\author{Harold Sultan}
\address{Department of Mathematics\\Columbia University\\
New York\\NY 10027}
\email{HSultan@math.columbia.edu}
\date{\today}

\begin{abstract}
We prove that in CAT(0) spaces a quasi-geodesic is Morse if and only if it is contracting.  Specifically, in our main theorem we prove that for $X$ a CAT(0) space and $\gamma \subset X$ a quasi-geodesic, the following four statements are equivalent:  (i) $\gamma$ is Morse, (ii) $\gamma$ is (b,c)--contracting, (iii) $\gamma$ is strongly contracting, and (iv) in every asymptotic cone $X_{\omega},$ any two distinct points in the ultralimit $\gamma_{\omega}$ are separated by a cutpoint.  As a corollary, we provide a converse to the usual Morse stability lemma in the CAT(0) setting.  In addition, as a warm up we include an alternative proof of the fact, originally proven in Behrstock-Dru\c tu \cite{behrstockdrutu}, that in CAT(0) spaces Morse quasi-geodesics have at least quadratic divergence.
\end{abstract}
\maketitle

\tableofcontents

\section{Introduction and Overview} \label{sec:intro}
In the course of studying metric spaces one is frequently interested in families of geodesics which admit \emph{hyperbolic type} properties, or properties exhibited by geodesics in hyperbolic space which are not exhibited by geodesics in Euclidean space.  In the geometric group theory literature there are various well studied examples of such hyperbolic type properties including being Morse, being contracting, having cutpoints in the asymptotic cone, and having at least quadratic divergence.  Specifically, such studies have proven fruitful in analyzing right angled Artin groups \cite{behrstockcharney}, Teichm\"uller space \cite{behrstock,brockfarb,brockmasur,bmm,mosher}, the mapping class group \cite{behrstock}, CAT(0) spaces \cite{behrstockdrutu,bestvina,charney}, and Out($F_{n}$) \cite{algomkfir} amongst others (See for instance \cite{drutumozessapir,drutusapir,kapovitchleeb,osin,mm1}). 

A Morse geodesic $\gamma$ is defined by property that all quasi-geodesics $\sigma$ with endpoints on $\gamma$ remain within a bounded distance from $\gamma.$  A  strongly contracting geodesic has the property that metric balls disjoint from the geodesic have nearest point projections onto the geodesic with uniformly bounded diameter.   The \emph{divergence} of a geodesic measures the inefficiency of detour paths.  More formally, divergence along a geodesic is defined as the growth rate of the length of detour paths connecting sequences of pairs of points on a geodesic, where the distance between the pairs of points is growing linearly while the detour function is forced to avoid linearly sized metric balls centered along the geodesic between the pairs of points.    

It is an elementary fact that in hyperbolic space all quasi-geodesics are Morse, strongly contracting, and have exponential divergence.  On the other end of the spectrum, in product spaces such as Euclidean spaces of dimension two and above, there are no Morse or strongly contracting quasi-geodesics, and all quasi-geodesics have linear divergence.  Relatedly, there are no cutpoints in any asymptotic cones of product spaces, whereas all asymptotic cones of a $\delta$-hyperbolic spaces are $\R$-trees, and hence any two distinct points are separated by a cutpoint.  

In this paper we will explore the close relationship between the aforementioned hyperbolic type properties of quasi-geodesics in CAT(0) spaces.  The following theorem is a highlight of the paper:
\begin{repthm} $\textbf{\ref{thm:main}.}$ \emph{(Main Theorem)
Let $X$ be a CAT(0) space and $\gamma \subset X$ a quasi-geodesic.  Then the following are equivalent:
\begin{enumerate}
\item $\gamma$ is (b,c)--contracting,
\item $\gamma$ is strongly contracting,
\item $\gamma$ is Morse, and
\item In every asymptotic cone $X_{\omega},$ any two distinct points in the ultralimit $\gamma_{\omega}$ are separated by a cutpoint. 
\end{enumerate}
In particular, any of the properties listed above implies that $\gamma$ has at least quadratic divergence.}
\end{repthm}

Theorem \ref{thm:main} should be considered in the context of related theorems in \cite{bestvina,charney,drutumozessapir,kapovitchleeb}.  Specifically, in \cite{kapovitchleeb} it is shown that periodic geodesics with superlinear divergence have at least quadratic divergence.   In \cite{drutumozessapir} it is shown that properties (3) and (4) in Theorem \ref{thm:main} are equivalent for arbitrary metric spaces.  In \cite{bestvina} it is shown that in proper CAT(0) spaces a geodesic which is the axis of a hyperbolic isometry is strongly contracting if and only if the geodesic fails to bound a half plane.  In \cite{charney} it is shown that geodesics with superlinear lower divergence are equivalent to strongly contracting geodesics.  The proof of Theorem \ref{thm:main} relies on careful applications of CAT(0) geometry and asymptotic cones.  

In \cite{kapovitchleeb,bestvina} it is shown that in proper CAT(0) spaces, periodic geodesics with superlinear divergence in fact have at least quadratic divergence.  Generalizing this result, in \cite{behrstockdrutu} it is shown that in CAT(0) spaces Morse quasi-geodesics have at least quadratic divergence.  As a warmup for Theorem \ref{thm:main}, in this paper we provide an alternative proof of this latter generalization.  To be sure, Theorem \ref{thm:main} itself also provides an alternative proof of the same result.
\begin{repthm} $\textbf{\ref{thm:morsequad}.}$ 
 \cite{behrstockdrutu} \emph{A Morse quasi-geodesic in a CAT(0) space has at least quadratic divergence.}
\end{repthm}

Additionally, in this paper we write down an explicit proof of the following generalization of the well known Morse stability lemma.  While the lemma is implicit for instance in \cite{behrstock,drutumozessapir}, there is no recorded proof for the following version of the lemma in the literature.  Accordingly, we include an explicit proof in this paper.
\begin{replem} $\textbf{\ref{lem:morse}.}$
\emph{Let $X$ be a geodesic metric space and $\gamma \subset X$ a (b,c)--contracting quasi-geodesic.  Then $\gamma$ is Morse.  Specifically, if $\sigma$ is a (K,L) quasi-geodesic with endpoints on $\gamma,$ then $d_{Haus}(\gamma,\sigma)$ is uniformly bounded in terms of only the constants $b,c,K,L.$}
\end{replem} 

The proof of Lemma \ref{lem:morse} is based on a similar proof in \cite{algomkfir} dealing with the special case where $\gamma$ is a strongly contracting geodesic.

Moreover, as a corollary of Theorem \ref{thm:main} we highlight the following converse to the aforementioned Morse stability Lemma:
\begin{repcor} $\textbf{\ref{cor:converse}.}$ 
\emph{Let $X$ be a CAT(0) space and $\gamma \subset X$ a Morse quasi-geodesic.  Then $\gamma$ is strongly contracting.}
\end{repcor}

The plan for the paper is as follows.  Section \ref{sec:back} provides background notation, definitions, and results used in the paper.  Section \ref{sec:theorems} includes the proof of Lemma \ref{lem:morse} and Theorems \ref{thm:morsequad} and \ref{thm:main}.  Section \ref{sec:applications} considers applications of Theorem \ref{thm:main} to the study of quasi-geodesics in CAT(0) spaces.  Finally, Section \ref{sec:future} closes with questions for future consideration.
 \subsection*{Acknowledgements}
$\\$
$\indent$
I want to express my gratitude to my advisors Jason Behrstock and Walter Neumann for their extremely helpful advice and insights throughout my research, and specifically regarding this paper.  

\section{Background}\label{sec:back}
\begin{defn} (quasi-geodesic) \label{defn:quasi} A \emph{(K,L) quasi-geodesic} $\gamma \subset X$ is the image of a map $\gamma:I \ra X$ where $I$ is a connected interval in $\R$ (possibly all of $\R$) such that $\forall s,t \in I$ we have the following quasi-isometric inequality:
$$\frac{|s-t|}{K}-L \leq  d_{X}(\gamma(s),\gamma(t)) \leq K|s-t|+L $$ 
We refer to the quasi-geodesic $\gamma(I)$ by $\gamma,$ and when the constants $(K,L)$ are not relevant omit them.   
\end{defn}

An arbitrary quasi-geodesic in any geodesic metric space can be replaced by a continuous rectifiable quasi-geodesic by replacing the quasi-geodesic with a piecewise geodesic path connecting consecutive integer valued parameter points of the original quasi-geodesic.  It is clear that this replacement process yields a continuous rectifiable quasi-geodesic which is in a bounded Hausdorff neighborhood of the original quasi-geodesic.  When doing so will not affect an argument, by replacement if necessary we will assume quasi-geodesics are continuous and rectifiable.  One upshot of the assumption of continuous quasi-geodesics is that the for $\gamma,\sigma$ quasi-geodesics, the distance function $\psi(t)=d(\gamma(t),\sigma)$ is continuous.  More generally, for non-continuous quasi-geodesics this distance function can have jump discontinuities controlled by the constants of the quasi-geodesics.  Throughout, for $\gamma$ any continuous and rectifiable path, we will denote its length by $|\gamma|.$  

The following definition of Morse (quasi-)geodesics has roots in the classical paper \cite{morse}:
\begin{defn}(Morse) A (quasi-)geodesic $\gamma$ is called a \emph{Morse (quasi-)geodesic} if every $(K,L)$-quasi-geodesic with endpoints on $\gamma$ is within a bounded distance from $\gamma,$ with the bound depending only on the constants $K,L.$  In the literature, Morse (quasi-)geodesics are sometimes referred to as \emph{stable quasi-geodesics}.
\end{defn} 

The following generalized notion of contracting quasi-geodesics can be found for example in \cite{behrstock,brockmasur}, and is based on a slightly more general notion of (a,b,c)--contraction found in \cite{mm1} where it serves as a key ingredient in the proof of the hyperbolicity of the curve complex.
\begin{defn} (contracting quasi-geodesics) A (quasi-)geodesic $\gamma$ is said to be \emph{(b,c)--contracting} if $\exists$ constants $0<b\leq 1$ and $0<c$ such that $\forall x,y \in X,$ $$d_{X}(x,y)<bd_{X}(x,\pi_{\gamma}(x))  \implies d_{X}(\pi_{\gamma}(x),\pi_{\gamma}(y))<c.$$
For the special case of a (b,c)--contracting quasi-geodesic where $b$ can be chosen to be $1,$ the quasi-geodesic $\gamma$ is called \emph{strongly contracting.}
\end{defn}  

\begin{defn}(Divergence) Let $\gamma:(-\infty,\infty)\ra X$ be a bi-infinite (quasi-)geodesic in $X.$  The \emph{divergence along $\gamma$} is defined to be the growth rate of the function $$d_{X \setminus B_{r}(\gamma(0))}(\gamma(-r),\gamma(r))$$ with respect to $r,$ where $r \in \N.$  More generally, given a sequence of (quasi-)geodesic segments $\gamma_{n}\subset X$ such that the lengths of the (quasi-)geodesic segments grow proportionally to a linear function, we can similarly define \emph{the divergence along the sequence $\gamma_{n}$} to be the growth rate of the lengths of detour functions for the sequence of (quasi-)geodesics $\gamma_{n}.$  In the literature there are various closely related definitions of divergence, see \cite{drutumozessapir} for details.
\end{defn} 

\subsection{Asymptotic Cones}
Let $(X, d)$ be a geodesic metric space, $\omega$ a non-principal ultrafilter, $(x_{n})$ a sequence of observation points in $X,$ and $(s_{n})$ a sequence of scaling constants such that $\lim_{\omega} s_{n} \ra \infty.$  Then the \emph{asymptotic cone}, $Cone_{\omega}(X,(x_{n}), (s_{n}))$, is the metric space consisting of equivalence classes of sequences $(y_{n})$ satisfying $\lim_{\omega}\frac{d(x_{n},y_{n})}{s_{n}} < \infty,$ where two such sequences $(y_{n})$ and $(y'_{n}),$ represent the same point if and only if $\lim_{\omega}\frac{d(y_{n},y'_{n})}{s_{n}} \ra 0.$  

Given any sequence of subsets $A_{n} \subset X,$ and asymptotic cone $X_{\omega},$ the \emph{ultralimit} $A_{\omega}$ is defined to be the subset of the cone $X_{\omega}$ with representative sequences $(z_{n})$ such that $\{n | z_{n} \in A_{n} \} \in \omega.$  When the choices of scaling constants and base points are not relevant, we denote the asymptotic cone by $X_{\omega}.$  Elements of asymptotic cones will be denoted $x_{\omega}$ with representative sequences denoted $(x_{i}).$   

The following theorem of \cite{drutumozessapir} characterizing Morse geodesics in terms of the asymptotic cone has application in this paper:
\begin{thm}  \label{thm:cutpoint} \cite{drutumozessapir} $\gamma$ is a Morse quasi-geodesic if and only if in every asymptotic cone $X_{\omega},$ every pair of distinct points in the ultralimit $\gamma_{\omega}$ are separated by a cutpoint.
\end{thm}


\subsection{CAT(0) geometry} CAT(0) spaces are metric spaces defined by the property that triangles are no ``fatter'' than the corresponding comparison triangles in Euclidean space.  In particular, using this property one can prove the following lemma, see \cite[Section II.2]{bridson} for details. 

\begin{lem} Let $X$ be a CAT(0) space.
\begin{enumerate} \label{lem:cat}
\item [C1:] (Projections onto convex subsets)  Let $C$ be a convex subset, complete in the induced metric, then there is a well defined distance non-increasing nearest point projection map $\pi_{C}: X \ra C.$  In particular, $\pi_{C}$ is continuous.  In this paper we will be interested in the special case of $C=\gamma$ a geodesic. 
\item [C2:] (Convexity) Let $c_{1}:[0,1] \ra X$ and $c_{2}:[0,1] \ra X$ be any pair of geodesics parameterized proportional to arc length.  Then the following inequality holds for all $t\in [0,1]:$
$$d(c_{1}(t),c_{2}(t)) \leq (1-t) d(c_{1}(0),c_{2}(0)) + td(c_{1}(1),c_{2}(1))  $$
\item[C3:] (Unique geodesic space) For any points $x,y \in X$ there is a unique geodesic segment $\gamma$ connecting $x$ and $y,$ which we denote by $[x,y].$  In particular, CAT(0) spaces are geodesic metric spaces.  
\end{enumerate}
\end{lem}  

Some caution must be taken in considering quasi-geodesics in CAT(0) spaces.  In fact, even in $\R^{2}$ it is easy to construct examples of quasi-geodesics onto which nearest point projections are not even coarsely well defined.  Nonetheless, for the classes of Morse and (b,c)--contracting quasi-geodesics this cannot occur.  The underlying point is that for any two points on a Morse or (b,c)--contracting quasi-geodesic, the unique geodesic connecting the points is contained in a bounded tubular neighborhood of the quasi-geodesic.  For Morse quasi-geodesics, this is part of the definition, while for (b,c)--contracting quasi-geodesics, this follows from Lemma \ref{lem:morse}.  As a consequence, we will see that Morse and (b,c)--contracting bi-infinite quasi-geodesics behave similarly to geodesics.  Throughout, we will be careful to point out when we are dealing with geodesics and quasi-geodesics, respectively.  
\section{Proof of Theorems} \label{sec:theorems}
As a warm up for Theorem \ref{thm:main}, we begin this section by giving an alternative proof of the fact that Morse quasi-geodesics in CAT(0) spaces have at least quadratic divergence.  This result was originally proven in \cite{behrstockdrutu}.  The present alternative proof is inspired by similar methods in \cite{kapovitchleeb} and follows immediately from the following lemma.  For the sake of simplifying the exposition, in Lemma \ref{lem:morsequad} we consider the special case of $\gamma$ a geodesic rather than a quasi-geodesic.  Hence properties in Lemma \ref{lem:cat} can be applied.  Nonetheless, below we will show that the current form of the lemma suffices to prove Theorem \ref{thm:morsequad} concerning quasi-geodesics.  

\begin{lem} \label{lem:morsequad}
Let $X$ be a CAT(0) space, and $\gamma$ a geodesic.  If for every asymptotic cone $X_{\omega},$ any two distinct points in the ultralimit $\gamma_{\omega}$ are separated by a cutpoint, then $\gamma$ has at least quadratic divergence.  Similarly, the same result holds for the case of $\{\gamma_{n}\}$ a sequence of geodesic segments in $X$ with lengths growing proportionally to a linear function.
\end{lem}

\begin{figure}[htpb]
\centering
\includegraphics[height=4 cm]{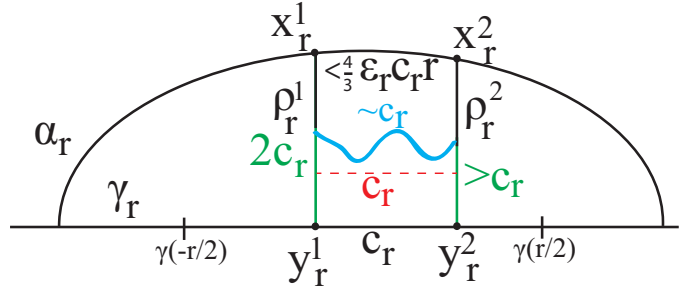}
\caption{In CAT(0) spaces subquadratic divergence implies the existence of an asymptotic cone $X_{\omega}$ in which distinct points in the ultralimit of the geodesic are not separated by a cutpoint.} \label{fig:morse1}
\end{figure}

\begin{proof}
We will prove the first statement in the Lemma.  The similar statement follows by the same argument.

By contradiction.  That is, assume $\gamma$ has subquadratic divergence.  By definition, for each $r\in \N,$ there is a continuous rectifiable detour path $\alpha_{r}$ connecting $\gamma(-r)$ and $\gamma(r)$ while remaining outside the ball $B_{r}(\gamma(0)),$ such that $|\alpha_{r}| \leq \epsilon_{r}r^{2}$ where the function $\epsilon_{r}$ satisfies $\lim_{r \ra \infty} \epsilon_{r} \ra 0.$
Fix a sequence $\{c_{r}\}_{r\in\N}$ such that:
\begin{enumerate}
\item $4c_{r} \leq r,$
\item $\lim_{r \ra \infty} c_{r} \ra \infty,$ and
\item $\lim_{r \ra \infty} c^{2}_{r}\epsilon_{r} \ra 0.$ 
\end{enumerate}
For example, set $c_{r} = \min \{ \epsilon^{-1/3}_{r}, \frac{r}{4} \}.$  

For each $r,$ let $n \in \{0,1,...,\lfloor \frac{r}{c_{r}} \rfloor \},$ and fix $z^{n}_{r} \in \alpha_{r}$ such that $z^{n}_{r} \in \pi^{-1}_{\gamma}(\gamma(-r/2 + nc_{r})).$  Since the total length of $\alpha_{r}$ is at most $ \epsilon_{r}r^{2},$ it follows that for some $m,$ the distance on $\alpha_{r}$ between $z^{m}_{r}$ and $z^{m+1}_{r}$ is at most 

$$\frac{\epsilon_{r}r^{2}}{\lfloor \frac{r}{c_{r}} \rfloor}  \leq \frac{\epsilon_{r}r^{2}}{ \frac{r-c_{r}}{c_{r}}} =  \frac{\epsilon_{r}c_{r}r^{2}}{r-c_{r}}  \leq  \frac{\epsilon_{r}c_{r}r^{2}}{r-\frac{r}{4}} =\frac{4\epsilon_{r}c_{r}r}{3}.$$

Set $x^{1}_{r}=z^{m}_{r},$ $x^{2}_{r}=z^{m+1}_{r},$ and $y^{i}_{r}=\pi_{\gamma}(x^{i}_{r}).$  By construction, $d(x^{1}_{r},x^{2}_{r}) \leq \frac{4}{3}\epsilon_{r}c_{r}r$ while $d(y^{1}_{r},y^{2}_{r})=c_{r}.$  Let $\rho^{i}_{r}:[0,1] \ra X$ be a geodesic parameterized proportional to arc length joining $y^{i}_{r}=\rho^{i}_{r}(0)$ and $x^{i}_{r}=\rho^{i}_{r}(1).$  See Figure \ref{fig:morse1}.  Note that by construction since $y_{r}^{i} \in \gamma[-r/2,r/2]$ and $x_{r}^{i} \in \alpha_{r},$ it follows that 

$$|\rho^{i}_{r}| \geq \frac{r}{2} \geq 2c_{r}.$$  Consider the function $\psi_{r}(t)=d(\rho^{1}_{r}(t),\rho^{2}_{r}(t)).$  Note that $\psi_{r}(0)=c_{r}$ and $\psi_{r}(1) \leq \frac{4}{3}\epsilon_{r}c_{r}r.$  CAT(0) convexity (Lemma \ref{lem:cat} property C2) implies that
$$ \psi_{r}\left(\frac{2c_{r}}{|\rho^{1}_{r}|}\right) \leq  \left(1-\frac{2c_{r}}{|\rho^{1}_{r}|}\right)c_{r} + \frac{8c_{r}}{3|\rho^{1}_{r}|}\epsilon_{r}c_{r}r  \leq c_{r} + \frac{16}{3}c^{2}_{r}\epsilon_{r}.$$
Since  $\lim_{r \ra \infty} c^{2}_{r}\epsilon_{r} \ra 0,$ for large enough $r$ we can assume $d(\rho^{1}_{r}\left(\frac{2c_{r}}{|\rho^{1}_{r}|}\right) ,\rho^{2}_{r}\left(\frac{2c_{r}}{|\rho^{1}_{r}|}\right))$ is arbitrarily close to $c_{r}.$ 

Since $y^{2}_{r}$ is a nearest point projection of $x^{2}_{r}$ onto $\gamma,$ it follows that $|\rho^{2}_{r}| \leq |\rho^{1}_{r}| + \frac{4}{3}\epsilon_{r}c_{r}r.$  Since $\lim_{r \ra \infty} c^{2}_{r}\epsilon_{r} \ra 0$ and $\lim_{r \ra \infty} c_{r} \ra \infty,$ in particular $\lim_{r \ra \infty} c_{r}\epsilon_{r} \ra 0.$  Hence, for sufficiently large $r$ we can assume $c_{r}\epsilon_{r}\leq \frac{3}{8}.$  Then we have the following inequality:
$$|\rho^{2}_{r}| \leq |\rho^{1}_{r}| + \frac{4}{3}\epsilon_{r}c_{r}r \leq |\rho^{1}_{r}| + \frac{1}{2}r \leq |\rho^{1}_{r}| +  |\rho^{1}_{r}| = 2|\rho^{1}_{r}|.$$
Running the same argument with the roles of $\rho^{1}_{r}$ and $ \rho^{2}_{r}$ reversed, it follows that $$ \frac{1}{2}|\rho^{1}_{r}| \leq  |\rho^{2}_{r}| \leq  2|\rho^{1}_{r}|.$$
In particular, $d(y^{2}_{r}, \rho^{2}_{r}\left(\frac{2c_{r}}{|\rho^{1}_{r}|} \right))$ is at most $4c_{r}$ and at least $c_{r}.$      

Putting things together, on the one hand we have a geodesic segment $[y^{1}_{r},y^{2}_{r}] \subset \gamma$ of length $c_{r}.$  While on the other hand we have a piecewise geodesic path $$\sigma_{r } = [y^{1}_{r},\rho^{1}_{r}\left(\frac{2c_{r}}{|\rho^{1}_{r}|}\right) ] \bigcup [\rho^{1}_{r}\left(\frac{2c_{r}}{|\rho^{1}_{r}|}\right),\rho^{2}\left(\frac{2c_{r}}{|\rho^{1}_{r}|}\right)] \bigcup [\left(\frac{2c_{r}}{|\rho^{1}_{r}|}\right),y^{2}_{r}],$$ of total length arbitrarily close to at most $7c_{r}.$  Moreover, note that by construction we can bound from below the distance between the geodesics $[y^{1}_{r},y^{2}_{r}]$ and $[\rho^{1}_{r}\left(\frac{2c_{r}}{|\rho^{1}_{r}|}\right),\rho^{2}_{r}\left(\frac{2c_{r}}{|\rho^{1}_{r}|}\right)].$  Specifically, it follows that the distance $$d([\rho^{1}_{r}\left(\frac{2c_{r}}{|\rho^{1}_{r}|}\right),\rho^{2}_{r}\left(\frac{2c_{r}}{|\rho^{1}_{r}|}\right)],[y^{1}_{r},y^{2}_{r}])$$ 
is at least arbitrarily close to $c_{r}.$  Consider the asymptotic cone $Cone_{\omega}(X,(y^{1}_{r}),(c_{r})).$  In this asymptotic cone, the distinct points $y^{1}_{\omega},y^{2}_{\omega}$ in the ultralimit $\gamma_{\omega}$ are not separated by a cutpoint due to the path $\sigma_{\omega}$ connecting them. This completes the proof.
\end{proof}

Using Lemma \ref{lem:morsequad} in conjunction with Theorem \ref{thm:cutpoint}, proven in \cite{drutumozessapir}, we provide an alternative proof of the following Theorem, originally proven in \cite{behrstockdrutu}:
\begin{thm}
\label{thm:morsequad} \cite{behrstockdrutu} Let $\gamma$ be a Morse quasi-geodesic in a CAT(0) space $X,$ then $\gamma$ has at least quadratic divergence.
\end{thm}

\begin{proof}
Given a Morse quasi-geodesic $\gamma,$ construct a sequence of geodesic segments $\gamma'_{n}$ connecting the points $\gamma(-n)$ and $\gamma(n).$  By the Morse property, all the geodesic segments $\gamma'_{n}$ are contained in a uniformly bounded Hausdorff neighborhood of $\gamma.$  By Theorem \ref{thm:cutpoint}, in any asymptotic cone $X_{\omega},$ any distinct points in $\gamma_{\omega}$ are separated by a cutpoint.  However, since the sequence of geodesics $\gamma'_{n}$ are in a uniformly bounded Hausdorff neighborhood of $\gamma$ it follows that in any asymptotic cone $X_{\omega},$ any distinct points in $\gamma'_{\omega}$ are similarly separated by a cutpoint.  Applying Lemma \ref{lem:morsequad} to the sequence of geodesic segments $\gamma'_{n},$ it follows that the sequence of geodesic segments has quadratic divergence.  However, since the quasi-geodesic $\gamma$ and sequence of geodesic segments $\gamma_{n}'$ are in a bounded Hausdorff neighborhood of each other they have the same order of divergence.
\end{proof}

With the end goal of proving Theorem \ref{thm:main}, presently we write down a proof generalizing the well known Morse stability lemma.  The usual Morse stability lemma states that a strongly contracting geodesic is Morse.  Presently, we show that the stability lemma holds for (b,c)--contracting quasi-geodesics.  To be sure, the proof of Lemma \ref{lem:morse} is closely modeled on the proof of the usual Morse stability lemma, e.g. as in \cite{algomkfir}.  

\begin{lem} \label{lem:morse}
Let $X$ be a geodesic metric space and $\gamma \subset X$ a (b,c)--contracting quasi-geodesic.  Then $\gamma$ is Morse.  Specifically, if $\sigma$ is a (K,L) quasi-geodesic with endpoints on $\gamma,$ then $d_{Haus}(\gamma,\sigma)$ is uniformly bounded in terms of only the constants $b,c,K,L.$
\end{lem} 

\begin{proof}
Since $\gamma$ is (b,c)--contracting, in particular the nearest point projection $\pi_{\gamma}$ is coarsely well defined.  Set $D= \max\{K, L, 1\}$, $A=\frac{2(1+cD)}{b},$ and $R = \max\{d(\gamma,\sigma) \mid t \in \R\}.$  Without loss of generality we can assume $R >A.$  Since we wish to show that $\sigma$ is in a bounded neighborhood of $\gamma,$ by replacement if necessary we can assume $\sigma$ is a continuous rectifiable quasi-geodesic.  

Let $[s_{1}, s_{2}]$ be any maximal connected subinterval in the domain of $\sigma$ such that $\forall s \in [s_{1}, s_{2}],$ we have $d(\sigma(s),\gamma) \geq A.$ Since $\sigma$ is continuous, we can subdivide the interval $[s_{1}, s_{2}]$ such that $s_{1} = r_{1}, . . . , r_{m}, r_{m+1} = s_{2}$ where $|\sigma(r_{i},r_{i+1})| = \frac{Ab}{2}$ for $i \leq m$ and $|\sigma(r_{m},r_{m+1})|  \leq \frac{Ab}{2}.$  Hence, 
\begin{eqnarray}\label{eq1} |\sigma(s1,s2)| \geq \frac{mAb}{2}.
\end{eqnarray}

Fix $P_{i} \in \pi_{\gamma}(\sigma(r_{i})).$ Then since $d(\sigma(r_{i}), P_{i}) \geq A$ and $d(\sigma(r_{i}),\sigma(r_{i+1})) \leq \frac{Ab}{2}<Ab,$ by (b,c)--contraction, $d(P_{i}, P_{i+1}) < c.$  Therefore $d(P_{1}, P_{m+1}) < c(m+1).$   It follows that $$d(\sigma(s_{1}),\sigma(s_{2})) < 2(A+L)+c(m+1).$$  Note that since we are not assuming $\gamma$ is a continuous quasi-geodesic, the distance function $d(\sigma(t),\gamma)$ can have jump discontinuities of $L.$ Using the fact that $\sigma$ is a quasi-geodesic, it follows that 
\begin{eqnarray}
\label{eq2} |\sigma(s1,s2)| \leq D(d(\sigma(s_{1}),\sigma(s_{2})) ) + D\leq D(2A+2L+cm+c+1).
\end{eqnarray}

Combining inequalities \ref{eq1} and \ref{eq2}, after some manipulation we obtain 

$$ m < \frac{D(2A+2L+c+1)}{\frac{Ab}{2}-cD} = D(2A+2L+c+1).$$

Thus, $\forall s \in [s_{1}, s_{2}]$ we have the following inequality:
\begin{eqnarray*}
d(\sigma(s),\gamma) &\leq& d(\sigma(s),\sigma(s_{2})) + d(\sigma(s_{2}),\gamma) \\
&\leq& |\sigma[s_{1},s_{2}]| + A +L\\
&\leq&  D(2A+2L+cm+c+1) +A +L \\
&<&  D(2A+2L+c\left(D(2A+2L+c+1)\right)+c+1) +A +L \\
\end{eqnarray*}
Since the constants $A,D$ are defined in terms of the constants $b,c,K,L,$ the lemma follows. 
\end{proof}

Using Lemma \ref{lem:morse}, we will prove the following theorem:

\begin{thm}\label{thm:main}
Let $X$ be a CAT(0) space and $\gamma \subset X$ a (K,L)--quasi-geodesic.  Then the following are equivalent:
\begin{enumerate}
\item $\gamma$ is (b,c)--contracting
\item $\gamma$ is (1,c)--contracting, (or strongly contracting)
\item $\gamma$ is Morse, and
\item In every asymptotic cone $X_{\omega},$ any two distinct points in the ultralimit $\gamma_{\omega}$ are separated by a cutpoint. 
\end{enumerate}
In particular, any of the properties listed above implies that $\gamma$ has at least quadratic divergence.
\end{thm}
\begin{rem} Theorem \ref{thm:main} stated for CAT(0) spaces also holds for \emph{rough CAT(0) spaces} or \emph{rCAT(0) spaces}, as defined in \cite{roughcat}.  rCAT(0) spaces are defined similarly to CAT(0) spaces, however in place of requiring that triangles are no fatter than corresponding comparison triangles in Euclidean space, we relax this inequality to allow for a fixed additive error.  In particular, the class of rCAT(0) spaces strictly includes CAT(0) spaces and all Gromov-hyperbolic spaces.  In fact, since in \cite{roughcat} it is shown that rCAT(0) spaces satisfy a coarse version of property [C3] in Lemma \ref{lem:cat}, modulo appropriate coarse modifications, the proof of Theorem \ref{thm:main} carries through for rCAT(0) spaces.
\end{rem}

\begin{figure}[htpb] 
\centering
\includegraphics[height=4 cm]{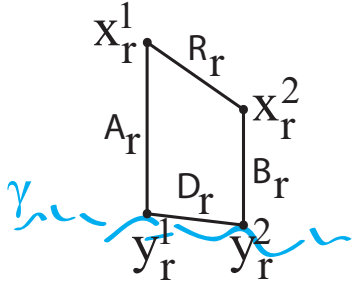}
\caption{In a CAT(0) space, assuming a quasi-geodesic $\gamma$ is not (1,c)--contracting implies it is not Morse.}\label{fig:1ccont}
\end{figure}

\begin{proof}
$\\$
\indent
$(2) \implies (1)$: This follows immediately from the definitions.  $(1)\implies(3)$: This is precisely Lemma \ref{lem:morse}.  $(3)\implies(4)$: This is precisely Theorem \ref{thm:cutpoint}, proven in \cite{drutumozessapir}.
$\\$
\indent
In the remainder of the proof we will prove $(4)\implies(2)$: By contradiction.  That is, assuming $\gamma$ is not (1,c)--contracting we will show that there is an asymptotic cone $X_{\omega}$ such that distinct points in the ultralimit $\gamma_{\omega}$ are not separated by a cutpoint.  Since $\gamma$ is not (1,c)--contracting, it follows that for all $r\in \N,$ we can make the following choices satisfying the stated conditions:
\begin{enumerate}
\item [(i)] Fix points $x^{1}_{r} \in X \setminus \gamma,$ and $y^{1}_{r} \in \pi_{\gamma}(x^{1}_{r})$ such that $d(x^{1}_{r},y^{1}_{r})=A_{r}$ and
\item [(ii)] Fix points $x^{2}_{r} \in X \setminus \gamma,$ and a point $y^{2}_{r} \in \pi_{\gamma}(x^{2}_{r})$ such that $d(x^{1}_{r},x^{2}_{r})=R_{r} < A_{r},$ and  $d(y^{1}_{r},y^{2}_{r}) = D_{r},$ for some $D_{r}\geq r.$  Set $B_{r}=d(x^{2}_{r},y^{2}_{r}).$  
\end{enumerate}
Let $\rho^{i}_{r}:[0,1] \ra X$ be a geodesic parameterized proportional to arc length joining $y^{i}_{r}=\rho^{i}_{r}(0)$ and $x^{i}_{r}=\rho^{i}_{r}(1).$  See Figure \ref{fig:1ccont} for an illustration of the situation.  Note we are not assuming the nearest point projection maps $\pi_{\gamma}$ are even coarsely well defined, but instead are simply picking elements of the set of nearest points subject to certain restrictions guaranteed by the negation of (1,c)--contraction.  In fact, we cannot have assumed that $y^{2}_{r}$ could have been chosen such that $d(y^{1}_{r},y^{2}_{r}) = r,$ as the nearest point projection map onto quasi-geodesics need not be continuous.  Moreover, it is possible that $x^{1}_{r}$ and $x^{2}_{r}$ are even the same point.

Since $d(y^{1}_{r},y^{2}_{r}) = D_{r},$ it follows that $A_{r}+R_{r}+B_{r} \geq D_{r}.$  Moreover, since $R_{r} < A_{r}$ and $B_{r} \leq R_{r}+ A_{r},$ it follows that $A_{r} > \frac{D_{r}}{4}.$  Fix $t=\frac{D_{r}}{4A_{r}} \in (0,1).$  Additionally, since $B_{r} < 2A_{r}$ it follows that $|[y^{2}_{r},\rho^{2}_{r}(t)]| < \frac{D_{r}}{2}.$
  
Since $A_{r} > D_{r}/4,$ the ratio $\frac{D_{r}}{A_{r}} \in (0,4),$ and hence there exists some subsequence such that $\frac{D_{r}}{A_{r}}$ converges.
$\\$
\textbf{Case 1:} There exists some subsequence such that $\frac{D_{r}}{A_{r}} \ra \epsilon \ne 0.$
$\\$
\begin{figure}[htpb] 
\centering
\includegraphics[height=4 cm]{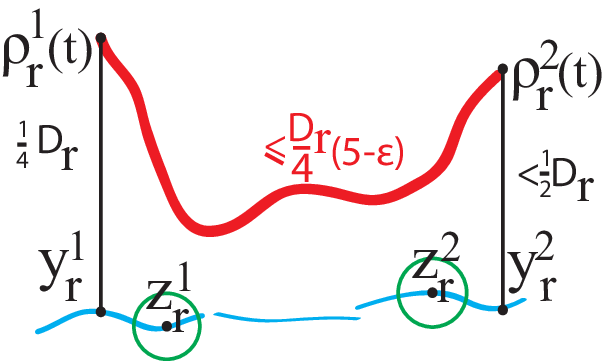}
\caption{Case (1) of the proof of Theorem \ref{thm:main}.}\label{fig:case1}
\end{figure}

\indent
CAT(0) convexity  (Lemma \ref{lem:cat} property C2) applied to the geodesics $\rho^{i}_{r}$ implies:
\begin{eqnarray*} \label{eq:convex}
d(\rho^{1}_{r}(t),\rho^{2}_{r}(t)) &\leq&  \left(1-\frac{D_{r}}{4A_{r}} \right) D_{r} + \frac{D_{r}R_{r}}{4A_{r}} \\
&\leq& D_{r} -\frac{D^{2}_{r}}{4A_{r}} D_{r} + \frac{D_{r}}{4} \leq   \frac{D_{r}}{4} \left(5- \frac{D_{r}}{A_{r}} \right).
\end{eqnarray*}

For large enough values of $r$ in the convergent subsequence, it follows that $d(\rho^{1}_{r}(t),\rho^{2}_{r}(t))$ is arbitrarily close to $\frac{D_{r}}{4} \left(5-\epsilon \right).$

Let $z^{1}_{r}$ be a point on $\gamma$ between $y^{1}_{r}$ and $y^{2}_{r}$ such that $d(y^{1}_{r},z^{1}_{r})$ is in the range $[\frac{\epsilon D_{r}}{28}, \frac{\epsilon D_{r}}{28}+L].$  Similarly, let $z^{2}_{r}$ be a point on $\gamma$ between $y^{1}_{r}$ and $y^{2}_{r}$ such that $d(y^{2}_{r},z^{2}_{r})$ is in the range $[\frac{\epsilon D_{r}}{28}, \frac{\epsilon D_{r}}{28}+L].$   Since $\rho^{i}_{r}$ are geodesics minimizing the distance from a fixed point to $\gamma,$ it follows that $\rho^{i}_{r}$ are disjoint from the interiors of the metric balls $B(z^{i}_{r}, \frac{\epsilon D_{r}}{56}).$

Moreover, by construction, for large enough values of $r$ in the convergence subsequence, the geodesic $[\rho^{1}_{r}(t),\rho^{2}_{r}(t)]$ is disjoint from either the metric ball $B(z^{1}_{r}, \frac{\epsilon D_{r}}{56})$ or the metric ball $B(z^{2}_{r}, \frac{\epsilon D_{r}}{56}).$  For if not, then 
\begin{eqnarray*}
|[\rho^{1}_{r}(t),\rho^{2}_{r}(t)]| &\geq& d(\rho^{1}_{r}(t),\{B(z^{1}_{r},\frac{\epsilon D_{r}}{56}),B(z^{2}_{r},\frac{\epsilon D_{r}}{56}) \}) + d(B(z^{1}_{r},\frac{\epsilon D_{r}}{56}),B(z^{2}_{r},\frac{\epsilon D_{r}}{56})) \\
&\geq& \left( \frac{D_{r}}{4}- \frac{\epsilon D_{r}}{56} \right)+ \left( D_{r} -6\frac{\epsilon D_{r}}{56} \right) = \frac{D_{r}}{4}\left( 5-\frac{\epsilon}{2} \right).
 \end{eqnarray*}
However, this contradicts the fact that $d(\rho^{1}_{r}(t),\rho^{2}_{r}(t))$ is arbitrarily close to $\frac{D_{r}}{4} \left(5-\epsilon \right).$  On the other hand, if for large enough values of $r$ in the convergence subsequence, the geodesic $[\rho^{1}_{r}(t),\rho^{2}_{r}(t)]$ is disjoint from the metric ball $B(z^{i}_{r}, \frac{\epsilon D_{r}}{56}),$ then we will construct an asymptotic cone in which distinct points on $\gamma_{\omega}$ are not separated by a cutpoint, thus completing the proof in this case.  

Specifically, let $\omega$ be a non-principal ultrafilter such that the set of values of $r$ in the convergence subsequence are an element of $\omega.$  Then consider the asymptotic cone $Cone_{\omega}(X,(y^{1}_{r}),(D_{r})).$  In this asymptotic cone, the points $v^{\pm}_{\omega}$ in the intersection of $\gamma_{\omega}$ and the metric ball $B(z^{i}_{\omega}, \frac{\epsilon D_{r}}{56}) $ are not separated by a cutpoint due to the existence of a path $[v^{+}_{\omega},z^{i}_{\omega}] \cup [z^{i}_{\omega},v^{-}_{\omega}]$ connecting them in the interior of the ball $B(z^{i}_{r},\frac{\epsilon D_{r}}{56}),$ as well as the path connecting them outside the ball $B(z^{i}_{r},\frac{\epsilon D_{r}}{56})$ given by the union of paths $$[v^{-}_{\omega},y^{1}_{\omega}] \cup [y^{1}_{\omega},\rho^{1}_{\omega}(t)] \cup [\rho^{1}_{\omega}(t),\rho^{2}_{\omega}(t)] \cup [\rho^{2}_{\omega}(t),y^{2}_{\omega}] \cup [y^{2}_{\omega},v^{+}_{\omega}].$$  See figure \ref{fig:case1} for an illustration of the proof in Case (1). 

\textbf{Case 2:} There exists some subsequence such that  $\frac{B_{r}}{D_{r}} \ra 0.$
$\\$
\begin{figure}[htpb] 
\centering
\includegraphics[height=5 cm]{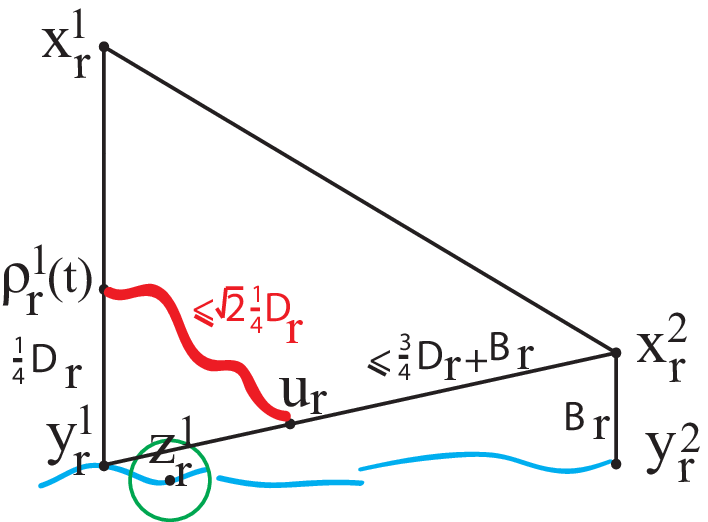}
\caption{Case (2) of the proof of Theorem \ref{thm:main}.}\label{fig:case4}
\end{figure}

Let $\sigma_{r}:[0,1] \ra X$ be a geodesic parameterized proportional to arc length joining $y^{1}_{r}=\sigma_{r}(0)$ and $x^{2}_{r}=\sigma_{r}(1).$  By the triangle inequality, $|\sigma_{r}|$ is in the range $[D_{r}-B_{r},D_{r}+B_{r}].$  

Consider the triangle in $X$ with vertices $(y^{1}_{r},x^{2}_{r},x^{1}_{r}),$ and let the comparison triangle in Euclidean space have vertices $(\overline{y^{1}_{r}},\overline{x^{2}_{r}},\overline{x^{1}_{r}}),$  Since $R_{r} < A_{r},$ it follows that the angle between the sides $[\overline{x^{2}_{r}},\overline{y^{1}_{r}}]$ and $[\overline{x^{1}_{r}}, \overline{y^{1}_{r}}],$ is less than $\frac{\pi}{2}.$  Let $\overline{u}_{r}$ denote the point in $[\overline{y^{1}_{r}}, \overline{x^{2}_{r}}],$ such that $d(\overline{y^{1}_{r}},\overline{u_{r}})=\frac{D_{r}}{4}.$  Elementary Euclidean trigonometry implies that $d(\overline{\rho^{1}_{r}(t)},\overline{u_{r}}) < \frac{\sqrt{2}D_{r}}{4}.$  Hence, by the CAT(0) property, it follows that $d(\rho^{1}_{r}(t),u_{r})< \sqrt{2} \epsilon D_{r}.$ 

Note that $d(u_{r},x^{2}_{r})\leq \frac{3 D_{r}}{4} +B_{r},$ and hence $d(u_{r},y^{2}_{r})\leq \frac{3 D_{r}}{4} +2B_{r}.$  Putting things together, it follows that $d(\rho^{1}_{r}(t),y^{2}_{r})<  \frac{\sqrt{2}D_{r}}{4}+\frac{3 D_{r}}{4} +2B_{r}.$  

As in Case (1), let $z^{1}_{r}$ be a point on $\gamma$ between $y^{1}_{r}$ and $y^{2}_{r}$ such that $d(y^{1}_{r},z^{1}_{r})$ is in the range $[\frac{(2-\sqrt{2}) D_{r}}{16}, \frac{(2-\sqrt{2}) D_{r}}{16}+L].$  Again as in Case (1), note that $\rho^{1}_{r}$ is disjoint from the interior of the metric balls $B(z^{1}_{r}, \frac{(2-\sqrt{2}) D_{r}}{32}).$  

Furthermore, for large enough values of $r$ in the convergence subsequence, the geodesic $[\rho^{1}_{r}(t),y^{2}_{r}]$ is also disjoint from the metric ball $B(z^{1}_{r},\frac{(2-\sqrt{2}) D_{r}}{32}).$  For if not, then 
\begin{eqnarray*}
|[\rho^{1}_{r}(t),y^{2}_{r}]| &\geq& d(\rho^{1}_{r}(t), B(z^{1}_{r},\frac{(2-\sqrt{2})D_{r}}{32}) ) + d(B(z^{1}_{r},\frac{(2-\sqrt{2}) D_{r}}{32}), y^{2}_{r}) \\
&\geq& \left( \frac{D_{r}}{4}- \frac{(2-\sqrt{2}) D_{r}}{32} \right)+ \left( D_{r} -3\frac{(2-\sqrt{2}) D_{r}}{32}\right) \geq D_{r} + \frac{\sqrt{2}D_{r}}{8}.
\end{eqnarray*}

However, in conjunction with the assumption of the case, this contradicts the fact that $d(\rho^{1}_{r}(t),y^{2}_{r})$ is at most $\frac{3D_{r}}{4}+ \frac{\sqrt{2}D_{r}}{4} +2B_{r}.$  On the other hand, if for large enough values of $r$ the geodesic $[\rho^{1}_{r}(t),y^{2}_{r}]$ is disjoint from the metric ball $B(z^{1}_{r}, \frac{(2-\sqrt{2})\epsilon D_{r}}{32}),$ then as in Case (1), in the asymptotic cone $Cone_{\omega}(X,(y^{1}_{r}),(D_{r}))$ we can find distinct points on $\gamma_{\omega}$ that are not separated by a cutpoint.  This completes the proof in Case (2).  See figure \ref{fig:case4} for an illustration of the proof in Case (2). 
$\\$
\textbf{Case 3:} We are not in Cases (1) or (2): 
$\\$
\indent
Since we are not in Case (2), by passing to a subsequence if necessary we can assume that the ratio $\frac{B_{r}}{D_{r}}$ either converges to $\epsilon' >0$ or diverges to infinity.  In the former case, set $\epsilon=\min(\frac{1}{4},\epsilon'),$ and in the latter case set $\epsilon=\frac{1}{4}.$  Set $s=\frac{\epsilon D_{r}}{A_{r}}.$  By construction $s \in (0,1).$   

Let $\tau_{r}:[0,1] \ra X$ be a geodesic parameterized proportional to arc length joining $x^{2}_{r}=\tau_{r}(0)$ and $x^{1}_{r}=\tau_{r}(1).$  Similarly, let $\sigma_{r}:[0,1] \ra X$ be a geodesic parameterized proportional to arc length joining $y^{2}_{r}=\sigma_{r}(0)$ and $x^{1}_{r}=\sigma_{r}(1).$  By construction, $|\sigma_{r}|$ is in the range $[A_{r},A_{r}+D_{r}].$  Since we are not in Case (1), it follows that $|[\sigma_{r}(0),\sigma_{r}(s)]|$ is arbitrarily close to $\epsilon D_{r}.$  Moreover, CAT(0) convexity (Lemma \ref{lem:cat} property C2) applied to the geodesics $\rho^{1}_{r}$ and $\sigma_{r}$ immediately implies $d(\rho^{1}_{r}(s),\sigma_{r}(s))$ is bounded above by $D_{r}.$

Consider the triangle in $X$ with vertices $(x^{1}_{r},x^{2}_{r},y^{2}_{r}),$ and let the comparison triangle in Euclidean space have vertices $(\overline{x^{1}_{r}},\overline{x^{2}_{r}},\overline{y^{2}_{r}}),$  As in Case (1), since $R_{r} < A_{r},$ it follows that the angle between the sides $[\overline{x^{1}_{r}},\overline{y^{2}_{r}}]$ and $[\overline{x^{2}_{r}}, \overline{y^{2}_{r}}],$ is less than $\frac{\pi}{2}.$  Let $\overline{w_{r}}$ denote the point in $[\overline{y^{2}_{r}}, \overline{x^{2}_{r}}],$ such that $d(\overline{y^{2}_{r}},\overline{w_{r}})=\epsilon D_{r}.$  Note that since $|[\sigma_{r}(0),\sigma_{r}(s)]|$ is arbitrarily close to $\epsilon D_{r},$ elementary Euclidean trigonometry implies that $d(\overline{\sigma_{r}(s)},\overline{w_{r}})$ is at most arbitrarily close to $\sqrt{2} \epsilon D_{r}.$  Hence, by the CAT(0) property, it follows that $d(\sigma_{r}(s),w_{r})$ is at most arbitrarily close to $\sqrt{2} \epsilon D_{r}.$  Putting things together, it follows that $d(\rho^{1}_{r}(s),w_{r})$ is at most arbitrarily close to $D_{r}+ \sqrt{2}\epsilon D_{r}.$  

As in Case (1), let $z^{1}_{r}$ be a point on $\gamma$ between $y^{1}_{r}$ and $y^{2}_{r}$ such that $d(y^{1}_{r},z^{1}_{r})$ is in the range $[\frac{(2-\sqrt{2})\epsilon D_{r}}{16}, \frac{(2-\sqrt{2})\epsilon D_{r}}{16}+L].$  Similarly, let $z^{2}_{r}$ be a point on $\gamma$ between $y^{1}_{r}$ and $y^{2}_{r}$ such that $d(y^{2}_{r},z^{2}_{r})$ is in the range $[\frac{(2-\sqrt{2})\epsilon D_{r}}{16}, \frac{(2-\sqrt{2})\epsilon D_{r}}{16}+L].$  For large enough values of $r$ in the convergence subsequence, the geodesic $[\rho^{1}_{r}(s),w_{r}]$ is disjoint from either the metric ball $B(z^{1}_{r},\frac{(2-\sqrt{2})\epsilon D_{r}}{32})$ or the metric ball  $B(z^{2}_{r},\frac{(2-\sqrt{2})\epsilon D_{r}}{32}).$  For if not, then 
\begin{eqnarray*}
|[\rho^{1}_{r}(s),w_{r}]| &\geq& d(\rho^{1}_{r}(t), \{B(z^{1}_{r},\frac{(2-\sqrt{2})\epsilon D_{r}}{32}), B(z^{2}_{r},\frac{(2-\sqrt{2})\epsilon D_{r}}{32})\}) \\
&+& d(B(z^{1}_{r},\frac{(2-\sqrt{2})\epsilon D_{r}}{32}), B(z^{2}_{r},\frac{(2-\sqrt{2})\epsilon D_{r}}{32})) \\
&+& d(w_{r},\{B(z^{1}_{r},\frac{(2-\sqrt{2})\epsilon D_{r}}{32}), B(z^{2}_{r},\frac{(2-\sqrt{2})\epsilon D_{r}}{32})\})\\
&\geq& \left( \epsilon D_{r}- \frac{(2-\sqrt{2})\epsilon D_{r}}{32} \right)+ \left( D_{r} -6\frac{(2-\sqrt{2})\epsilon D_{r}}{32}\right) +  \left( \epsilon D_{r}- \frac{(2-\sqrt{2})\epsilon D_{r}}{32} \right) \\
 &>& D_{r} +  \frac{3\epsilon D_{r}}{2}.  \end{eqnarray*}

However, this is a contradiction to the fact that $d(\rho^{1}_{r}(s),w_{r})$ is at most arbitrarily close to $D_{r} + \sqrt{2}\epsilon D_{r}.$  On the other hand, if for large enough values of $r$ in the convergence subsequence, the geodesic $[\rho^{1}_{r}(s),w_{r}]$ is disjoint from the metric ball $B(z^{i}_{r}, \frac{(2-\sqrt{2})\epsilon D_{r}}{32}),$ then as in Case (1), the asymptotic cone $Cone_{\omega}(X,(y^{1}_{r}),(D_{r}))$ contains distinct points of $\gamma_{\omega}$ not separated by a cutpoint, thereby completing the proof in the final case and hence completing the proof of $(4) \implies (2).$  See figure \ref{fig:case3} for an illustration of the proof in Case (3). 
 
\begin{figure}[htpb] 
\centering
\includegraphics[height=5 cm]{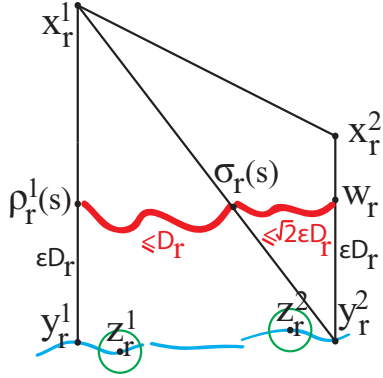}
\caption{Case (3) of the proof of Theorem \ref{thm:main}.}\label{fig:case3}
\end{figure}

Finally, the ``in particular'' clause of the theorem follows from Theorem \ref{thm:morsequad}.  
 \end{proof}
 
\section{Applications of Theorem \ref{thm:main}}
\label{sec:applications}
In this section we organize some applications of the Theorem \ref{thm:main}.  First, as an immediate consequence of Theorem \ref{thm:main}, we highlight the following which provides a converse to the usual Morse stability Lemma for CAT(0) spaces: 
\begin{cor} \label{cor:converse} 
Let $X$ be a CAT(0) space and $\gamma \subset X$ a Morse quasi-geodesic.  Then $\gamma$ is strongly contracting.
\end{cor}

The completion of Teichm\"uller space equipped with the Weil Petersson metric, $\mathcal{\overline{T}}_{WP}(S),$ is a CAT(0) metric space which has been an object of interest in recent years especially within the circle of ideas surrounding the resolution of Thurston's Ending Lamination Conjecture.  As $\mathcal{\overline{T}}_{WP}(S)$ is CAT(0), Theorem \ref{thm:main} has application to the study of quasi-geodesics in this space.  Specifically, in \cite{behrstock,bmm,sultanfinest} among others, families of quasi-geodesics with various of the hyperbolic type properties recorded in Theorem \ref{thm:main} are studied.  In particular, in \cite{behrstock}, it is shown that for $\gamma$ a quasi-geodesic in $\mathcal{\overline{T}}_{WP}(S)$ with \emph{bounded combinatorics} (see \cite{bmm} for the definition), any distinct points in the ultralimit of the quasi-geodesic in any asymptotic cone are separated by a cutpoint.  On the other hand, in \cite{bmm} it is shown that quasi-geodesics in $\mathcal{\overline{T}}_{WP}(S)$ with bounded combinatorics are (b,c)--contracting.  Similarly, implicitly in \cite{behrstock} as well as in \cite{sultanfinest} it is shown that a more general class of geodesics generalizing those with bounded combinatorics also has the property that any distinct points in the ultralimit of the quasi-geodesics in any asymptotic cone are separated by a cutpoint.   Putting things together, in conjunction with Theorem \ref{thm:main}, we have the following corollary strengthening the aforementioned results in \cite{behrstock,bmm}:

\begin{cor}
Let $\gamma$ be a quasi-geodesic in $\mathcal{\overline{T}}_{WP}(S)$ with bounded combinatorics (or more generally the generalization of bounded combinatorics studied in \cite{sultanfinest}), then $\gamma$ is strongly contracting.
\end{cor}

Note that \cite{bestvina} proves the special case of Corollary \ref{cor:bcimpstrong} where $\gamma$ is a psuedo-Anosov axis in $\mathcal{\overline{T}}_{WP}(S).$  More generally, as a special case of Theorem \ref{thm:main} the following corollary can be used for proving that a quasi-geodesic in a CAT(0) space is strongly contracting:
\begin{cor} \label{cor:bcimpstrong}
Let $\gamma$ be a quasi-geodesic in a CAT(0) space.  If for some constant $0<b\leq 1, \; c>0$, the quasi-geodesic $\gamma$ is (b,c)--contracting, then $\gamma$ is strongly contracting.
\end{cor}
  
\section{Questions}
\label{sec:future}
In light of Theorem \ref{thm:main}, we pose the following natural question:
\begin{qu} \label{qu:converse}
For $X$ a CAT(0) space and $\gamma \subset X$ a quasi-geodesic, under what conditions is the property of $\gamma$ having at least quadratic divergence equivalent to the other four properties in Theorem \ref{thm:main}?  Equivalently, when does the converse of Lemma \ref{lem:morsequad} hold?    
\end{qu}
If one considers the union of two copies of $R^{n}$ joined at a single point, it is easy find geodesics in such a CAT(0) space which have at least quadratic, in fact infinite, divergence, yet are not Morse.  On the other hand, if $\gamma$ is periodic then having quadratic divergence, in fact even superlinear divergence, by homogeneity implies that in every asymptotic cone any distinct points in the ultralimit $\gamma_{\omega}$ are separated by a cutpoint, \cite{drutumozessapir}.  

Furthermore, in light of Lemma \ref{lem:morse}, we pose the following question:
\begin{qu} \label{qu:cubic} As in Lemma \ref{lem:morse} asymptotic cones can be used to provide conditions for determining when a quasi-geodesic has at least quadratic divergence.  Can asymptotic cones be used to determine cubic divergence or any higher degree of polynomial divergence?
\end{qu}

Taking steps toward answering Question \ref{qu:cubic}, in \cite{sultanfinest} the author uses asymptotic cones to prove that the Teichm\"uller space equipped with the Weil Petersson metric of the once punctured genus two surface, $\mathcal{\overline{T}}_{WP}(S_{2,1}),$ has superquadratic divergence.

\bibliographystyle{amsalpha}

\end{document}